\documentclass[11pt,twoside]{amsart}
\usepackage{amsfonts}
\usepackage[latin9]{inputenc}
\usepackage{amsmath}
\usepackage{amssymb}
\usepackage{breqn}
\usepackage{url}
\usepackage{graphicx}
\usepackage[pdfstartview=FitH]{hyperref}
\usepackage{amsthm}
\usepackage{hyperref}
\usepackage{url}
\usepackage{cleveref}
\usepackage{verbatim}
\usepackage{bbm}

\crefformat{section}{\S#2#1#3} % see manual of cleveref, section 8.2.1
\crefformat{subsection}{\S#2#1#3}
\crefformat{subsubsection}{\S#2#1#3}

\newcommand{\B}{\operatorname{\mathbb{E}}}
\newcommand{\Sym}{\operatorname{Sym}}

\newcommand{\id}{\operatorname{id}}

\newcommand{\sides}{\operatorname{sides}}

\newcommand{\Span}{\operatorname{span}}
\newcommand{\type}{\operatorname{type}}

\newcommand{\bipartite}{\operatorname{bipartite}}

\makeatletter
\begin{document}
\newtheorem{theorem}{Theorem}[section]
\newtheorem{lemma}[theorem]{Lemma}
\newtheorem{claim}[theorem]{Claim}
\newtheorem{proposition}[theorem]{Proposition}
\newtheorem{corollary}[theorem]{Corollary}
\theoremstyle{definition}
\newtheorem{definition}[theorem]{Definition}
\newtheorem{observation}[theorem]{Observation}
\newtheorem{example}[theorem]{Example}
\newtheorem{remark}[theorem]{Remark}

\title{Property (T) for groups acting on affine buildings}

\author{Izhar Oppenheim}
\address{Department of Mathematics, Ben-Gurion University of the Negev, Be'er Sheva,  Israel}
\email{izharo@bgu.ac.il}

%\author{Izhar Oppenheim}
%\newcommand{\Addresses}{{% additional braces for segregating \footnotesize
  %\bigskip
  %\footnotesize

  %IZHAR OPPENHEIM, \textsc{Department of Mathematics, Ben-Gurion University of the Negev, Be'er Sheva 84105, Israel}\par\nopagebreak
  %\textit{E-mail address:} \texttt{izharo@bgu.ac.il}
%}}

\maketitle
\begin{abstract}
We prove that a group acting geometrically on a thick affine building has property (T).  A more general criterion for property (T) is given for groups acting on partite complexes.
\end{abstract} 
%\textbf{Mathematics Subject Classification (2010)}. Primary ,Secondary. \\
%\textbf{Keywords}. .

\section{Introduction}

 Recall that a locally compact group $G$ is said to act geometrically on a simplicial complex $X$,  if it acts on $X$ properly and cocompactly.

The aim of this note is to prove the following Theorem:
\begin{theorem}
\label{main thm}
Let $X$ be a thick affine building of dimension $\geq 2$, i.e., an affine building of dimension $\geq 2$ in which every codimension-one simplex is a face of at least three chambers.  If $G$ is a locally compact, unimodular group that acts geometrically on $X$,  then $G$ has property (T).
\end{theorem}

We note that for classical affine buildings,  i.e.,  affine buildings arising from a BN-pair of an algebraic group over a non-archimedean field,  this result can be deduced from the fact that higher rank algebraic groups and their lattices have property (T).  Also,  when $X$ is an exotic $\widetilde{A}_2$ building, this result readily follows from the so called Zuk's criterion (see \cite{Ballmann}).  However, when $X$ is an exotic building of type $\widetilde{C}_2$ or $\widetilde{G}_2$, the existing results in the literature require the building thickness to be sufficiently large and do not cover buildings of small thickness (see \cite[Table 2]{OppAve}). Part of the motivation for this note is the recently announced result by Titz Mite and Witzel \cite{TMW}, in which they construct a non-residually finite group acting properly and cocompactly on an exotic $\widetilde{C}_2$ building with small thickness.

We prove Theorem \ref{main thm} by establishing a general criterion for property (T) for groups acting geometrically on partite simplicial complexes - a result of independent interest.  Notably, we only require a weaker assumption concerning the properness of the action (see details below).

\section{Preliminaries}

\subsection{Angle between subspaces}

\begin{definition}
[Angle between subspaces]\cite[Definition 3.2, Remark 3.19]{Kassabov} Let $\mathcal{H}$ a Hilbert space and 
$U_1$ and $U_2$ be two closed subspaces in $\mathcal{H}$. The cosine of $\angle  (U_1,U_2)$ is defined as $0$ if $U_1\subseteq U_2 \; or \; U_2\subseteq U_1$ and otherwise as
$$
\cos \angle (U_1,U_2)=
\sup \{\vert \left\langle u_1,u_2\right\rangle\vert \; :\;\Vert u_i\Vert =1, u_i\in U_i, u_i\perp (U_1\cap U_2)\}.$$
\end{definition}

For a Hilbert space $\mathcal{H}$ and a closed subspace $U$,  we will denote $P_U$ to be the orthogonal projection on $U$.

%\begin{remark}
%There is an alternative definition of $\cos \angle (U_1,U_2)$ in the language of projections: For $U_1, U_2$ as above
%$$\cos \angle (U_1,U_2) = \Vert P_{U_1} P_{U_2} - P_{U_1 \cap U_2} \Vert.$$
%The proof of the equivalence between this definition and the one given above is straightforward and can be found in \cite[Lemma 9.5]{Deutsch}. 
%\end{remark}

The following result was proven in \cite{Kassabov}:
\begin{theorem}\cite[Theorem 5.1]{Kassabov}
\label{Kas Thm}
Let $\mathcal{H}$ be a Hilbert space and $U_0,...,U_n$ closed subspaces of $\mathcal{H}$.  Let $\mathcal{A} (U_0,...,U_n)$ to be the $(n +1) \times (n+1)$ matrix defined as 
$$\mathcal{A} (U_0,...,U_n) [i,j] = \begin{cases}
1 & i=j \\
- \cos \angle (U_i,U_j) &  i \neq j
\end{cases}.$$
If $\mathcal{A}$ is positive definite,  then for every $x \in \mathcal{H}$ it holds that 
$$\Vert x - P_{U_0 \cap ... \cap U_n} x \Vert^2 \leq  \textbf{d}_{U_0,...,U_n} \left( \mathcal{A} (U_0,...,U_n) \right)^{-1}  \textbf{d}_{U_0,...,U_n}^t,$$
where $\textbf{d}_{U_0,...,U_n}$ is the vector $\textbf{d}_{U_0,...,U_n} = \left(\Vert x - P_{U_0} x \Vert,  ... , \Vert x - P_{U_n} x \Vert \right)$.
\end{theorem}

\subsection{Random walks on bipatite finite graphs}
\label{rw on graph subsec}
Given a finite connected graph $(V,E)$,  we define $m : V \rightarrow \mathbb{N}$ to be the degree of $v$ for every $v \in V$.  For a non-empty subset $V' \subseteq V$,  we further define $m(V') = \sum_{v \in V'} m(v)$.   

The graph $(V,E)$ is called \textit{bipartite} if $V$ can be partitioned into two disjoint sets $V = S_1 \cup S_2$ called \textit{sides} such that for each $\lbrace u,v \rbrace \in E$, $\vert \lbrace u,v \rbrace \cap S_1 \vert =  \vert \lbrace u,v \rbrace \cap S_2 \vert =1$, i.e., each edge has exactly one vertex in each side.  We note that it follows that $m (S_1) = m(S_2) = \frac{1}{2} m(V)$.

We also define $\ell^2 (V,m)$ to be the space of functions $\phi: V \rightarrow \mathbb{C}$ with an inner-product
$$\langle \phi, \psi \rangle = \sum_{\lbrace v \rbrace \in V} = m(v) \phi (v) \overline{\psi (v)}.$$

The \textit{random walk} on $(V,E)$ as above is the operator $A: \ell^2 (V,m) \rightarrow \ell^2 (V,m)$ defined as  
$$(A \phi) (v) = \sum_{u \in V, \lbrace u, v \rbrace \in E} \frac{1}{m(v)} \phi (u).$$

We state without proof a few basic facts regarding the random walk operator:
\begin{enumerate}
\item With the inner-product defined above, $A$ is a self-adjoint operator and the eigenvalues of $A$ lie in the interval $[-1,1]$.
\item The space of constant functions is an eigenspace of $A$ with eigenvalue $1$ and if $(V,E)$ is connected all other the other eigenfunctions of $A$ have eigenvalues strictly less than $1$.
\item Assuming that the graph $(V,E)$ is connected it holds that this graph is bipartite if and only if $-1$ is an eigenvalue of $A$.
\end{enumerate}

Assuming that $(V,E)$ is bipartite with sides $S_1, S_2$,  we note that the spectrum of $A$ is symmetric: Explicitly, if $\lambda$ is an eigenvalue of $A$ with an eigenfunction $\phi$, then 
$$\phi ' (u) = 
\begin{cases}
\phi (u) & u \in S_1 \\
- \phi (u) & u \in S_2
\end{cases}$$
is an eigenfunction of $A$ with an eigenvalue $- \lambda$.  In particular $\phi = \mathbbm{1}_{S_1} - \mathbbm{1}_{S_2}$ is an eigenfunction with the eigenvalue $-1$.

We define the following averaging operators:
$M_{1}, M_{2} : \ell^2 (V, m) \rightarrow \mathbb{C}$: 
$$M_{i} \phi = \frac{1}{m (S_{i})} \sum_{u \in S_{i}} m(u) \phi (u).$$
We also define $M_{\sides} : \ell^2 (V, m) \rightarrow \ell^2 (V, m)$ by 
$$M_{\sides} \phi (u) = 
\begin{cases}
M_{1} \phi & u \in S_{1} \\
M_{2} \phi & u \in S_{2} 
\end{cases}.$$

We note that $M_{\sides}$ is the orthogonal projection  on the space of functions $\Span \lbrace \mathbbm{1}_V,   \mathbbm{1}_{S_i} - \mathbbm{1}_{S_j} \rbrace$.

We recall the following definition of spectral expansion:
\begin{definition}
Let $(V,E)$ be a finite connected graph and $0 \leq \lambda <1$ a constant. The graph $(V,E)$ is called a one-sided $\lambda$-spectral expander if the spectrum of $A$ is contained in $[-1, \lambda] \cup \lbrace 1 \rbrace$. 
\end{definition}
As noted above,  for a bipartite graph the spectrum of $A$ is symmetric with an eigenvalue $-1$ with an eigenfunction $\mathbbm{1}_{S_1} - \mathbbm{1}_{S_2} $.  Thus, for a connected bipartite graph $(V,E)$ it holds that the graph is a one-sided $\lambda$-spectral expander if and only if $\Vert A (I- M_{\sides}) \Vert \leq \lambda$. 

Given a Banach space $\B$, we define $\ell^2 (V, m ; \B)$ to be the Banach space of the functions $\phi : V \rightarrow \B$ with the norm
$$\Vert \phi \Vert^2 = \sum_{v \in V} m (v) \vert \phi (v) \vert_{\B}^2.$$
For a Hilbert space $\mathcal{H}$,  $\ell^2 (V, m ; \mathcal{H})$ is a Hilbert space with the inner-product
$$\langle \phi, \psi \rangle_{\ell^2 (V, m ; \mathcal{H})} = \sum_{v \in V} m (v) \langle \phi (v),  \psi (v) \rangle_{\mathcal{H}}.$$

Given a Banach space $\B$,  we define the operator $(A (I-M_{\sides})) \otimes \id_{\B} : \ell^2 (V,m ; \B) \rightarrow \ell^2 (V,m ; \B)$ and denote $\lambda_{(V,E),\bipartite}^{\B} = \Vert (A (I-M_{\sides})) \otimes \id_{\B} \Vert_{B(\ell^2 (V,m ; \B))}$.

\begin{proposition}
\label{rw eigenv bound prop}
Let $(V,E)$ be a finite connected bipartite graph and denote  $\lambda$ to be the second largest eigenvalue of the random walk $A: \ell^{2} (V,m) \rightarrow \ell^2 (V,m)$.  For every Hilbert space $\mathcal{H}$,  $\lambda_{(V,E),\bipartite}^{\mathcal{H}} \leq \lambda$.
\end{proposition}

\begin{proof}
Denote $\lambda$ to be the second largest eigenvalue of $A$.  We will show that for every $\phi \in \ell^2 (V, m ; \mathcal{H})$ with $\Vert \phi \Vert=1$,  it holds that   
$$\Vert (A (I-M_{\sides})) \otimes \id_{\mathcal{H}}  \phi\Vert \leq \lambda.$$
Let $\phi \in \ell^2 (V, m ; \mathcal{H})$ of norm $1$.  Since $V$ is finite,  the image of $\phi$,  $\phi (V)$ is contained in a finite dimensional Hilbert space of dimension $\leq \vert V \vert$.  Thus,  we can assume that $\phi \in  \ell^2 (V,m ; \mathbb{C}^{\vert V \vert})$ and it is enough to show that $\lambda_{(V,E),\bipartite}^{\mathbb{C}^{\vert V \vert}} \leq \lambda$.   By our assumption we can write $\phi = ( \phi_i)_{1 \leq i \leq \vert V \vert}$.  We note that by the definition of the norm,  it holds that 
$$\Vert \phi \Vert_{\ell^2 (V,m ; \mathbb{C}^{\vert V \vert})}^2 = \sum_{i=1}^{\vert V \vert} \Vert \phi_i \Vert_{\ell^2 (V,m)}^2.$$
Also,  
$$A (I-M_{\sides})) \otimes \id_{\mathbb{C}^{\vert V \vert}} \phi = (A (I-M_{\sides})) \phi_i )_{1 \leq i \leq \vert V \vert}.$$
Thus,  it is enough to prove that for every $i$, 
$$\Vert A (I-M_{\sides})) \phi_i \Vert_{\ell^2 (V,m)} \leq \lambda \Vert \phi_i \Vert_{\ell^2 (V,m)}$$
and this inequality holds since $I-M_{\sides}$ is a projection on the space $\left( \Span \lbrace \mathbbm{1}_V,   \mathbbm{1}_{S_i} - \mathbbm{1}_{S_j} \rbrace\right)^{\perp}$ and on this space all the eigenvalues of $A$ are $\leq \lambda$.  
\end{proof}

\begin{remark}
In the notation of the Proposition above,  it is not hard to show that it holds that $\lambda_{(V,E),\bipartite}^{\mathcal{H}} = \lambda$ for every non-trivial Hilbert space $\mathcal{H}$, i.e.,  for $\mathcal{H} \neq \lbrace 0 \rbrace$.  This fact is left for the reader and we will make no use of it in the sequel. 
\end{remark}

\subsection{Partite simplicial complexes}

Given a set $V$, an abstract simplicial complex $X$ with a vertex set $V$ is a family of subsets $X \subseteq 2^V$ such that if $\tau \in X$ and $\eta \subseteq \tau$, then $\eta \in X$. We will denote $X(k)$ to be the sets in $X$ of cardinality $k+1$.  A simplicial complex $X$ is called \textit{$n$-dimensional} if $X(n+1) = \emptyset$ and $X(n) \neq \emptyset$.  An $n$-dimensional simplicial complex $X$ is called \textit{pure} $n$-dimensional if for every $\tau$ in $X$, there is $\sigma \in X(n)$ such that $\tau \subseteq \sigma$.  

A pure $n$-dimensional simplicial complex $X$ is called \textit{gallery connected} if for every $\sigma, \sigma ' \in X(n)$, there is a finite sequence $\sigma_0,...,\sigma_l \in X(n)$ such that $\sigma = \sigma_0, \sigma ' = \sigma_l$ and for every $1 \leq i \leq l-1$, $\sigma_i \cap \sigma_{i+1} \in X(n-1)$.  For a gallery $\sigma_0,...,\sigma_l$,  the number $l$ is called the \textit{length of the gallery} and the \textit{gallery distance between $\sigma,  \sigma '$} is the length of the shortest gallery connecting $\sigma$ and $\sigma '$.  Below, we will always assume that $X$ is pure $n$-dimensional and gallery connected.

%Formally, $X(0)$ is the set of singletons of the form $\lbrace v \rbrace$ where $v$ is a vertex of $X$. Below, we will abuse the notation and denote $X(0)$ to also be the set of vertices of $X$ (writing $v \in X(0)$ instead of $\lbrace v \rbrace \in X(0)$).

%We define the following weight function $m: \bigcup_{k=0}^{n} \X (k) \rightarrow \mathbb{R} \cup \lbrace \infty \rbrace$ inductively as follows: 
%$$\forall \sigma \in \X (n), m(\sigma) =1,$$
%For $0 \leq k \leq n-1$ and $\tau \in X(k)$,
%$$m (\tau) = \sum_{\sigma \in \X (k+1), \tau \subseteq \sigma} m(\sigma).$$
%More explicitly,
%$$\forall \tau \in \X (k), m(\tau) = (n-k)! \vert \lbrace \sigma \in \X (n) : \tau \subseteq \sigma  \rbrace \vert.$$

Given a simplex $\tau \in X$, the \textit{link of $\tau$} is the subcomplex of $X$, denoted $X_\tau$, that is defined as
$$X_\tau = \lbrace \eta \in X : \tau \cap \eta = \emptyset, \tau \cup \eta \in X \rbrace.$$
%We define a weight function $m_\tau$ on $X_\tau$ by 
%$$m_\tau (\eta) = m (\tau \cup \eta).$$
%Below we will only be interested in the $1$-dimensional links of $X$: We note that if $\tau \in X(k)$ and $X$ is pure $n$-dimensional, then $X_\tau$ is pure $(n-k-1)$-dimensional. In particular, if $\tau \in X(n-2)$, then $X_\tau$ is a graph. Thus, we will refer to all links of the form $X_\tau$ where $\tau \in X(n-2)$ as \textit{the $1$-dimensional links of $X$}.  
%For $1$-dimensional links, the weight function is easy to describe: For a $1$-dimensional link $X_\tau$, we note that this weight function assigns the weight $1$ for each edge in the link $X_\tau$ and for each every $\lbrace u  \rbrace \in X_\tau (0)$, $m_\tau (\lbrace u \rbrace)$ is the valency of $\lbrace u \rbrace$ in $X_\tau$.

Given a pure $n$-dimensional simplicial complex $X$, we call $X$ \textit{partite}  if there is a partition of the vertex set $V$,  $V_0 \sqcup ...  \sqcup V_n = V$ such that for every $\sigma \in X(n)$ and every $0 \leq i \leq n$ it holds that $\vert \sigma \cap V_i \vert =1$, i.e., every $n$-dimensional simplex has exactly one vertex in each of the sets $V_0,...,V_n$.  A partite simplicial complex is also sometimes called \textit{colorable}, since we can think of the partition $V_0,...,V_n$ as a coloring of the vertex sets with $n+1$ colors, such that each $n$-dimensional simplex has vertices with all the colors (or equivalently, each $n$-dimensional simplex do not have two vertices with the same color).  For an $n$-dimensional partite simplicial complex, we define a type function $\type : X \rightarrow 2^{\lbrace 0,..., n\rbrace}$ by 
$$\type (\tau) = \lbrace i : \exists v \in \tau \cap V_i \rbrace.$$
Last,  note that if $X$ is a pure $n$-dimensional partite simplicial complex, then all the $1$-dimensional links of $X$ are bipartite graphs.

\section{Property (T) criterion for a group acting on a partite complex}

Here we will establish a general criterion for property (T) for a group acting on a partite simplicial complex.

We start by recalling the framework and some results from \cite{OppZukB}.

Let $X$ be a partite pure $n$-dimensional simplicial complex such that $X$ is gallery connected and the $1$-dimensional links of $X$ are connected finite graphs.  Also let $G$ be a locally compact, unimodular group acting on $X$ such that the action is cocompact and for every $\tau \in X(n-2) \cup X(n-1) \cup X(n)$ the subgroup stabilizing $\tau$, denoted $G_\tau$, is an open compact subgroup.  We also assume that the action of $G$ is type-preserving, i.e., for every $0 \leq k \leq n$ and every $\tau \in X(k)$, it holds for every $g \in G$ that $\type (\tau) = \type (g.\tau)$.  Last, let $\pi$ be a continuous unitary representation $\pi$ of $G$ on a Hilbert space $\mathcal{H}$.

\begin{remark}
\label{type preserving rmrk}
The assumption that the action is type preserving is in order to avoid some technical issues.  We note that this assumption is not very restrictive since every group $G$ acting on a partite simplicial complex $X$ that is gallery connected  has a finite index subgroup that is type preserving.  Indeed,  let $X$ be partite, gallery connected simplicial complex and $G$ a group acting on $X$.  Let $\sigma = \lbrace v_0,...,v_n \rbrace,  \sigma ' = \lbrace w_0,...,w_n \rbrace \in X(n)$ such that $\type (v_i) = \type (w_i) = i$.  By induction on the gallery distance between $\sigma$ and $\sigma '$, one can show that for every $g \in G$,  $\type (g.v_i) = \type (g.w_i)$.  It follows that when fixing $\sigma =  \lbrace v_0,...,v_n \rbrace$ as above,  the map $\Phi : G \rightarrow \Sym \lbrace 0,...,n \rbrace$ defined as $\Phi (g) (i) = \type (g^{-1} .v_i)$ is a group homomorphism and its kernel is a finite index subgroup of $G$  that is type preserving.
\end{remark}

%For every $n-2 \leq k \leq n$,  we choose a fundamental domain $D(k)$ of the action of $G$ on $X(k)$ such that the following holds: for every $n-2 \leq k_1 < k_2 \leq n$ and every $\tau \in D(k_1)$ there is $\sigma \in D (k_2)$ such that $\tau \subseteq \sigma$. 

We define $C (X(n), \pi)$ to be the space of maps $\phi : X(n) \rightarrow \mathcal{H}$ that are equivariant with respect to $\pi$, i.e., for every $\sigma \in X(n)$ and every $g \in G$, $\pi (g) \phi (\sigma) = \phi (g. \sigma)$.  We further define a norm on $C (X(n), \pi)$ by fixing a fundamental domain $D(n)$ for the action of $G$ on $X(n)$ and defining
$$\Vert \phi \Vert^2 = \left(  \sum_{\sigma \in D(n)} \frac{1}{\mu (G_\sigma )} \right)^{-1} \sum_{\sigma \in D(n)} \frac{1}{\mu (G_\sigma )} \vert \phi (\sigma) \vert^2,$$
where $G_\sigma$ is the subgroup stabilizing $\sigma$,  $\mu$ is the Haar measure of $G$ and $\vert . \vert$ is the norm of $\mathcal{H}$.   

For $\emptyset \neq \nu \subseteq \lbrace 0,....,n \rbrace$ define an equivalence relation $\sim_\nu$ on $X(n)$ as $\sigma \sim_\nu \sigma '$ if $\vert \sigma \cap \sigma ' \vert \geq \vert \nu \vert$ and there is $\tau \in X$ such that $\type (\tau) = \nu$ and $\tau \subseteq \sigma \cap \sigma ' $.  Note that by the assumption that the $1$-dimensional links of $X$ are finite graphs, it follows that for every $\emptyset \neq \nu \subseteq \lbrace 0, ...,n\rbrace$, $\vert \nu \vert = n$ and every $\sigma \in X(n)$,  the set $\lbrace \sigma ' \in X(n) : \sigma \sim_\nu \sigma ' \rbrace$ is finite. For $\emptyset \neq \nu \subseteq \lbrace 0, ...,n\rbrace$, $\vert \nu \vert = n$,  define a projection $P_\nu^\pi : C (X(n), \pi) \rightarrow C (X(n), \pi)$ by 
$$P_\nu^\pi \phi (\sigma) = \frac{1}{\vert \lbrace \sigma ' \in X(n) : \sigma \sim_\nu \sigma ' \rbrace \vert} \sum_{\sigma ' \sim_\nu \sigma} \phi (\sigma '),$$
(verifying that $P_\nu^\pi \phi$ is equivariant with respect to $\pi$ is straight-forward and left for the reader).  Denote 
$$C (X(n), \pi)_\nu = \lbrace \phi \in C( X(n), \pi) : \forall \sigma, \sigma ', \sigma \sim_\nu \sigma ' \Rightarrow \phi (\sigma ) = \phi (\sigma ') \rbrace,$$
and note that $P_\nu^\pi$ is a projection on $ C (X(n), \pi)_\nu$.

\begin{lemma}\cite[Lemma 3.5]{OppZukB}
\label{intersect lemma}
Let $X$, $G$ and $\pi$ be as above.  The space $\bigcap_{\nu \subseteq \lbrace 0,....,n \rbrace,  \vert \nu \vert = n} C (X(n), \pi)_\nu$ is the space of constant functions $\phi \equiv x_0$ where $x_0 \in \mathcal{H}^{\pi (G)}$.
\end{lemma}

We note that for every $\tau \in X(n-2)$, $X_\tau$ is a bipartite graph and we denote $\lambda_{\tau,   \bipartite}^{\mathcal{H}} = \lambda_{X_\tau,   \bipartite}^{\mathcal{H}}$. 
\begin{lemma}\cite[Lemma 3.8]{OppZukB}
Let $X$, $G$ and $\pi$ be as above.  For every $\nu, \nu ' \subseteq  \lbrace 0,....,n \rbrace,  \vert \nu \vert = \vert \nu ' \vert= n,  \nu \neq \nu '$,  denote
$$ \lambda_{\nu \cap \nu '}^{\mathcal{H}} =  \max_{\tau \in X(n-2)/G,   \type (\tau) = \nu \cap \nu '} \lambda_{\tau,   \bipartite}^{\mathcal{H}}.$$
Then it holds that 
$$\cos \angle (C (X(n), \pi)_\nu ,  C (X(n), \pi)_{\nu '}  ) \leq   \lambda_{\nu \cap \nu '}^{\mathcal{H}}.$$ 
\end{lemma}

We note that for every $\eta \in X(n-2)$,  $X_\eta$ is a connected bipartite graph and define $\lambda_\eta$ to be the second largest eigenvalue of the random walk on $X_\eta$.  For $0 \leq i,j \leq n,  i \neq j$,  we define
$$\lambda_{i,j} = \max_{\eta \in X(n-2) / G,  \type (\eta) = \lbrace 0,..., n \rbrace \setminus \lbrace i,j \rbrace} \lambda_{\eta}$$
and further define $\mathcal{A} (X)$ to be the $(n+1) \times (n+1)$ matrix 
$$\mathcal{A} (X) [i,j] = \begin{cases}
1 & i = j \\
- \lambda_{i,j} & i \neq j
\end{cases}.$$

\begin{corollary}
\label{lambda X coro}
Let $X$, $G$ and $\pi$ be as above.  Denote $U_i = C (X(n), \pi)_{\lbrace 0,...,n \rbrace \setminus \lbrace i \rbrace}$ and $\mathcal{A} (\pi) = \mathcal{A} (U_0,...,U_n)$ to be the cosine matrix as in Theorem \ref{Kas Thm},  i.e., 
$$\mathcal{A} (\pi) [i,j] =
\begin{cases}
1 & i = j \\
- \cos \angle (U_i, U_j) & i \neq j
\end{cases}.$$
Then for every $i,j$,  $\mathcal{A} (\pi) [i,j] \geq \mathcal{A} (X) [i,j]$.  

In particular,   if $\mathcal{A} (X)$ is positive definite and its smallest eigenvalue is $\lambda_X$,  then $\mathcal{A} (\pi)$ is positive definite and its smallest eigenvalue is $\geq \lambda_X$. 
\end{corollary}

\begin{proof}
By the above Lemma and Proposition \ref{rw eigenv bound prop} it follows for every $i \neq j$,  that $\mathcal{A} (\pi) [i,j] \geq \mathcal{A} (X) [i,j]$ as needed. The fact that the smallest eigenvalue of $\mathcal{A} (\pi)$ is $\geq \lambda_X$ follows from \cite[Proposition 2.8]{GRO}.
\end{proof}

For $x \in \mathcal{H}$, define $\phi_x : D(n) \rightarrow \mathcal{H}$ by 
$$\phi_x (\sigma) =\pi \left( \frac{ \mathbbm{1}_{G_{\sigma}}}{\mu (G_\sigma )} \right) x,  \forall \sigma \in D(n),$$
where $\mathbbm{1}_{G_{\sigma}}$ is the indicator function of $G_\sigma$.  Observe that $\phi_x (\sigma) \in \mathcal{H}^{\pi (G_\sigma)}$ for every $\sigma \in D(n)$.  Extend $\phi_x$ to $X(n)$ as follows: for every $g \in G$ and every $\sigma \in D(n)$, define 
$$\phi_x (g.\sigma) = \pi (g) \phi_x (\sigma) =   \pi \left( \frac{ \mathbbm{1}_{g G_{\sigma}}}{\mu (g G_\sigma )} \right) x =  \pi \left( \frac{ \mathbbm{1}_{g  G_{\sigma}}}{\mu (G_\sigma )} \right) x.$$

\begin{proposition} \cite[Proposition 3.3]{OppZukB}
The map $\phi_x$ is well-defined and equivariant, i.e., $\phi_x \in C(X(n), \pi)$.
\end{proposition} 

After this set-up, we are ready to state and prove our criterion for property (T):

\begin{theorem}
\label{criterion thm}
Let $X$ be a partite pure $n$-dimensional simplicial complex such that $X$ is gallery connected and the $1$-dimensional links of $X$ are connected finite graphs.  Also let $G$ be a locally compact, unimodular group acting on $X$ such that the action is cocompact and for every $\tau \in X(n-2) \cup X(n-1) \cup X(n)$ the subgroup stabilizing $\tau$, denoted $G_\tau$, is an open compact subgroup.  Assume that $\mathcal{A} (X)$ is positive definite,  then $G$ has property (T).
\end{theorem}

\begin{proof}
Recall that a group $G$ has property (T) if and only if every finite index group of $G$ has property (T).  Thus, by Remark \ref{type preserving rmrk} we can assume that $G$ is type preserving.

Define $K \subseteq G$ to be the set
$$K = \lbrace g \in G : \exists \sigma, \sigma ' \in D(n), \vert g. \sigma ' \cap \sigma  \vert \geq n \rbrace.$$
Note that $K$ is a compact symmetric set that generates $G$ (the last fact is due to the assumption that $X$ is gallery connected).  

Let $\lambda_X$ be the smallest eigenvalue of $\mathcal{A} (X)$.  We will show that for 
$$\varepsilon = \frac{1}{2 \left(1+ \sqrt{\frac{n+1}{\lambda_X}} \right)},$$
that $(K,\varepsilon)$ is a Kazhdan pair,  i.e.,  that for every continuous unitary representation $(\pi, \mathcal{H})$ with $\mathcal{H}^{\pi (G)} = \lbrace 0 \rbrace$,  it holds for every unit vector $x \in \mathcal{H}$ that
$$\sup_{g \in K} \vert \pi (g) x - x \vert \geq \varepsilon.$$

We will prove the contra-positive,  i.e.,  that given a  continuous unitary representation $(\pi, \mathcal{H})$,  if there exists a unit vector $x \in \mathcal{H}$ such that
$$\sup_{g \in K} \vert \pi (g) x - x \vert < \varepsilon,$$
then $\mathcal{H}^{\pi (G)} \neq  \lbrace 0 \rbrace$.  

Assume that there is a unit vector $x \in \mathcal{H}$ such that
$$\sup_{g \in K} \vert \pi (g) x - x \vert < \varepsilon.$$

For every $\sigma \in D (n)$ it holds that $G_\sigma \subseteq K$ and thus for every $\sigma \in D(n)$,
\begin{align*}
\vert x - \phi_x (\sigma) \vert = \vert x - \pi \left( \frac{ \mathbbm{1}_{G_{\sigma}}}{\mu (G_\sigma )} \right) x \vert \leq 
\frac{1}{ \mu (G_{\sigma})} \int_{g \in G_{\sigma}} \vert x - \pi (g) x \vert < \varepsilon.
\end{align*}
It follows that 
\begin{align*}
\Vert \phi_x \Vert \geq  \vert x \vert - \left(  \left(  \sum_{\sigma \in D(n)} \frac{1}{\mu (G_\sigma )} \right)^{-1} \sum_{\sigma \in D(n)} \frac{1}{\mu (G_\sigma )} \vert x- \phi_x (\sigma) \vert^2 \right)^{\frac{1}{2}} > 1- \varepsilon.
\end{align*}

Similarly,  let $\sigma \in D(n)$ and $\sigma '' \in X(n)$ such that $\vert \sigma \cap \sigma '' \vert =n$.  Then there is $g ' \in K$ and $\sigma ' \in D(n)$ such that $g ' . \sigma ' = \sigma ''$.  Note that $g' G_{\sigma '} \subseteq K$ and thus 
\begin{align*}
\vert x - \phi_x (\sigma ' ) \vert = \vert x - \pi (g' ) \pi \left( \frac{ \mathbbm{1}_{G_{\sigma ' }}}{\mu (G_{\sigma ' } )} \right) x \vert \leq 
\frac{1}{ \mu (G_{\sigma ' })} \int_{g \in G_{\sigma ' }} \vert x - \pi (g' g) x \vert < \varepsilon.
\end{align*}
It follows that for every $\nu \subseteq \lbrace 0,...,n \rbrace$,  $\vert \nu \vert = n$, every $\sigma \in D(n)$ and every $\sigma '' \sim_{\nu} \sigma$,
\begin{align*}
\vert \phi_x (\sigma) - \phi_x (\sigma ') \vert < 2 \varepsilon.
\end{align*}
Thus,  for every $\sigma \in D(n)$ and every $\nu \subseteq \lbrace 0,...,n \rbrace$,  $\vert \nu \vert = n$,
\begin{align*}
\vert \phi_x (\sigma) - P_\nu^\pi \phi_x (\sigma ) \vert < 2 \varepsilon,
\end{align*}
and also
$$\Vert \phi_x -  P_\nu^\pi \phi_x \Vert < 2 \varepsilon.$$

Denote $U_i = C (X(n), \pi)_{\lbrace 0,...,n \rbrace \setminus \lbrace i \rbrace}$ and $\mathcal{A} (\pi) = \mathcal{A} (U_0,...,U_n)$.  We showed that for every $0 \leq i \leq n$,  
$$\Vert \phi_x -  P_{U_i} \phi_x \Vert < 2 \varepsilon.$$
By Corollary \ref{lambda X coro},   $\mathcal{A} (\pi)$ is a positive definite matrix and its smallest eigenvalue is $\geq \lambda_X$.  Thus $(\mathcal{A} (\pi))^{-1}$ is a symmetric positive definite matrix and its largest eigenvalue is $\leq \frac{1}{\lambda_X}$.  By Theorem \ref{Kas Thm}  it holds that 
\begin{align*}
& \Vert \phi_x - P_{U_0 \cap ... \cap U_n} \phi_x \Vert^2 \leq \\
& \left(\Vert \phi_x  - P_{U_0} \phi_x  \Vert,  ... , \Vert \phi_x  - P_{U_n} \phi_x  \Vert \right)\left( \mathcal{A} (\pi) \right)^{-1}  \left(\Vert \phi_x  - P_{U_0} \phi_x  \Vert,  ... , \Vert \phi_x  - P_{U_n} \phi_x  \Vert  \right)^t \leq \\
& \frac{1}{\lambda_X} \sum_{i=0}^n \Vert \phi_x  - P_{U_i} \phi_x  \Vert^2 < \frac{1}{\lambda_X} \sum_{i=0}^n (2 \varepsilon)^2 = \frac{4 (n+1) \varepsilon^2}{\lambda_X}.
\end{align*}

Combining the inequalities above yields that 
\begin{align*}
\Vert  P_{U_0 \cap ... \cap U_n} \phi_x \Vert \geq \Vert \phi_x \Vert -  \Vert \phi_x - P_{U_0 \cap ... \cap U_n} \phi_x \Vert > \\
1  - 2 \varepsilon-  \sqrt{\frac{4 (n+1) \varepsilon^2}{\lambda_X}} = 
1 - 2 \varepsilon \left(1+ \sqrt{\frac{n+1}{\lambda_X}} \right) =^{\text{choice of } \varepsilon} 0,
\end{align*}
i.e.,  $\Vert  P_{U_0 \cap ... \cap U_n} \phi_x \Vert > 0$.  By Lemma \ref{intersect lemma},  there is $x_0 \in \mathcal{H}^{\pi (G)}$ such that $P_{U_0 \cap ... \cap U_n} \phi_x \equiv x_0$ and since $\Vert  P_{U_0 \cap ... \cap U_n} \phi_x \Vert > 0$,  it follows that $x_0 \neq 0$,  i.e., that $\mathcal{H}^{\pi (G)} \neq \lbrace 0 \rbrace$ as needed.
\end{proof}

\section{Property (T) for groups acting on affine buildings}

In order to apply Theorem \ref{criterion thm} for affine buildings,  we will need the following result:
\begin{lemma}\cite[Lemma 4.5]{GRO}
\label{pos def lemma}
For a thick affine building $X$ of dimension $\geq 2$,  $\mathcal{A} (X)$ is positive definite.
\end{lemma}

\begin{proof}[Proof of Theorem \ref{main thm}]
Let $G$ be a locally compact, unimodular group acting geometrically on a thick affine building of dimension $\geq 2$.  By Lemma \ref{pos def lemma},  $\mathcal{A} (X)$ is positive definite and thus by Theorem \ref{criterion thm},  $G$ has property (T).
\end{proof}

\bibliographystyle{alpha}
\bibliography{bibl}
%\Addresses
\end{document}